\documentclass[microtype]{gtpart}
\usepackage{graphicx,overpic}
\usepackage[mathscr]{eucal}
\usepackage[top=2.7cm, left=4cm, right=4cm, bottom=5cm]{geometry}
\usepackage{enumerate}
\usepackage{hyperref}  
\hypersetup{%
  bookmarksnumbered=true,%
  bookmarks=true,%
  colorlinks=black,%
  linkcolor=black,%
  citecolor=black,%
  filecolor=black,%
  menucolor=black,%
  pagecolor=black,%
  urlcolor=black,%
  pdfnewwindow=true,%
  pdfstartview=FitBH}
\usepackage{mathtools}

\newcommand{\ov}[1]{\overline{#1}}

\newcommand{\bh}{\mathbb H}
\newcommand{\bz}{\mathbb Z}
\newcommand{\bp}{\mathbb P}
\newcommand{\br}{\mathbb R}
\newcommand\Sp{\mathbb{S}}
\newcommand{\cO}{\mathcal O}
\newcommand{\cA}{\mathcal A}
\newcommand{\rp}{\br\bp}
\newcommand{\la}{\langle}
\newcommand{\ra}{\rangle}
\newcommand{\Lra}{\Longrightarrow}
\newcommand{\bs}{\backslash}
\newcommand{\wt}{\widetilde}
\newcommand{\wh}{\widehat}
\newcommand{\iso}{\cong}
\newcommand{\al}{\alpha}
\newcommand{\be}{\beta}
\newcommand{\si}{\sigma}

\newcommand{\lam}{\lambda}

\DeclareMathOperator{\im}{im}
\DeclareMathOperator{\sech}{sech}

\DeclareMathOperator{\PSL}{PSL}
\DeclareMathOperator{\SL}{SL}

\newcommand{\bndry}{\partial_\infty \pi_1}

\usepackage{thmtools, thm-restate}

\declaretheorem[name=Theorem,numberwithin=section]{Thm}
\newtheorem{thm}{Theorem}[section]
\newtheorem{Prop}[thm]{Proposition}
\newtheorem{Lem}[thm]{Lemma}
\newtheorem{Cor}[thm]{Corollary}

\theoremstyle{definition}

\theoremstyle{remark}
\newtheorem{Rem}[thm]{Remark}

\numberwithin{equation}{section}

\title{Basmajian's identity in higher Teichm\"uller-Thurston theory}

\author{Nicholas G.~Vlamis}
\givenname{Nicholas}
\surname{Vlamis}
\address{Department of Mathematics \\ University of Michigan \\ Ann Arbor, MI 48109}
\email{vlamis@umich.edu}
\urladdr{http://www.umich.edu/~vlamis}

\author{Andrew Yarmola}
\givenname{Andrew}
\surname{Yarmola}
\address{Mathematics Research Unit \\ University of Luxembourg \\ L-4364 Esch-sur-Alzette, Luxembourg}
\email{andrew.yarmola@uni.lu}
\urladdr{http://math.uni.lu/~yarmola/}

\begin{document}
\begin{abstract} We prove an extension of Basmajian's identity to $n$-Hitchin representations of compact bordered surfaces. For $n=3$, we show that this identity has a geometric interpretation for convex real projective structures analogous to Basmajian's original result. As part of our proof, we demonstrate that, with respect to the Lebesgue measure on the Frenet curve associated to a Hitchin representation, the limit set of an incompressible subsurface of a closed surface has measure zero. 
This generalizes a classical result in hyperbolic geometry.  Finally, we recall the Labourie-McShane extension of the McShane-Mirzakhani identity to Hitchin representations and note a close connection to Basmajian's identity in both the hyperbolic and the Hitchin settings.
\end{abstract}

\maketitle

\section{Introduction}

Let $\Sigma$ be a connected oriented compact surface with nonempty boundary whose double has genus at least 2.  Given a finite area hyperbolic metric $\si$ on $\Sigma$ such that $\partial \Sigma$ is totally geodesic, an \textit{orthogeodesic} in $(\Sigma, \si)$ is defined to be an oriented proper geodesic arc perpendicular to $\partial\Sigma$ at both endpoints; we denote the collection of all such arcs as $\cO(\Sigma, \sigma)$.  The \textit{orthospectrum} $|\cO(\Sigma,\sigma)|$ is the multiset of lengths of orthogeodesics counted with multiplicity.  
In \cite{Basmajian:1993eu}, Basmajian proved a remarkable identity on Teichm\"uller space that computes the length of the boundary as a sum over the orthospectrum:
\[
\ell_\si(\partial \Sigma) = \sum_{\ell \in |\cO(\Sigma,\sigma)|} 2\log\coth\left(\frac\ell2\right) \, ,
\]
where $\ell_\si(\partial \Sigma)$ denotes the length of $\partial \Sigma$ as measured in $\si$.  

In this paper, we formulate an extension of this identity to the setting of Hitchin representations (see \S\ref{sec:hitchin}) using Labourie's notion of cross ratios introduced in \cite{LabourieCross}.  Hyperbolic identities have played a surprising role in understanding the geometry  of the moduli space of Riemann surfaces (see \cite{mirzakhani}).  We expect these identities to continue to play a role in the study of higher Teichm\"uller-Thurston theory.  Before stating our main theorem, we need to introduce some notation and definitions, which will be elaborated on in \S\ref{sec:background}.

Recall that a Hitchin representation $\rho\co \pi_1(\Sigma) \to \PSL(n,\R)$ gives rise to a notion of length given by 
$$\ell_\rho(\gamma) = \log\left|\frac{\lambda_{\mathrm{max}}(\rho(\gamma))}{\lambda_{\mathrm{min}}(\rho(\gamma))}\right| \, ,$$
where $\lambda_{\mathrm{max}}(\rho(\gamma))$ and $\lambda_{\mathrm{min}}(\rho(\gamma))$ are the eigenvalues of maximum and minimum absolute value of $\rho(\gamma)$, respectively.

Assume that $\Sigma$ has $m$ boundary components and let $\cA = \{\al_1, \ldots, \al_m\}$ be a collection of primitive peripheral elements of $\pi_1(\Sigma)$ representing distinct boundary components oriented such that the surface is to the left.  We call such a collection a \textit{positive peripheral marking}.  Set $H_i = \la \al_i\ra$ and define the \textit{orthoset} to be the following disjoint union of double cosets:
\[
\cO(\Sigma, \cA) = \left(\bigsqcup_{1\leq i, j \leq m} H_i \bs \pi_1(\Sigma) / H_j\right) \setminus \left( \bigcup_{i=1}^m \{H_ieH_i\} \right) ,
\]
where $e\in \pi_1(\Sigma)$ is the identity.  The orthoset serves as an algebraic replacement for $\cO(\Sigma, \sigma)$. We demonstrate in \S\ref{sec:cosets} that there is a bijection between $\cO(\Sigma, \cA)$ and $\cO(\Sigma, \sigma)$ in the hyperbolic setting.

A \textit{cross ratio} on the Gromov boundary $\bndry(\Sigma)$ of $\pi_1(\Sigma)$ is a H\"older function defined on 
$$\bndry(\Sigma)^{4*} = \{(x,y,z,t) \in \bndry(\Sigma)^4 \colon x \neq t \text{ and } y\neq z\}$$
invariant under the diagonal action of $\pi_1(\Sigma)$, which satisfies several symmetry conditions (see \S\ref{sec:ratios} for a full definition).
In \cite{LabourieCross}, Labourie associates a cross ratio $B_\rho$ to a Hitchin representation $\rho$ of a closed surface. For surfaces with boundary, the same is done in \cite{Labourie:2009gb}. Define the function $G_\rho\co \cO(\Sigma, \cA) \to \br$ by 
$$G_\rho(H_igH_j) = \log B_\rho(\al_i^+, g\cdot \al_j^+, \al_i^-, g\cdot \al_j^-) \, ,$$
where $\al^+, \al^- \in \bndry(\Sigma)$ are the attracting and repelling fixed points  of $\al\in \pi_1(S)$, respectively. 
We can now state our main theorem:

\begin{restatable}{Thm}{basidentity}\label{thm:identity}{\rm (Basmajian's identity for Hitchin Representations)}
Let $\Sigma$ be an oriented compact connected surface  with $m>0$ boundary components whose double has genus at least 2. Let $\cA=\{\al_1, \ldots, \al_m\}$ be a positive peripheral marking. If $\rho$ is a Hitchin representation of $\pi_1(\Sigma)$, then 
$$\ell_\rho(\partial \Sigma) \, \, = \sum_{x\in \cO(\Sigma, \cA)} G_\rho(x) \,,$$
where $\ell_\rho(\partial \Sigma) = \sum_{i=1}^m \ell_\rho(\al_i).$  Furthermore, if $\rho$ is Fuchsian, this is Basmajian's identity.
\end{restatable}

We note that by cutting a surface along a simple closed geodesic, one can see that Basmajian's original identity and the Collar Lemma in hyperbolic geometry are intimately related.  Recently, Lee-Zhang \cite{LeeCollar} gave an extension of the Collar Lemma to the setting of Hitchin representations.  It would be interesting to understand the relationship between the Collar Lemma and Basmajian's identity in higher Teichm\"uller-Thurston theory.

In order to prove Theorem \ref{thm:identity}, we need to understand the measure of $\bndry(\Sigma)$ in the limit set of its double.  In \cite{labourie1}, Labourie defines a H\"older map $\xi_\rho\co \bndry(S) \to \bp\br^n$ for an $n$-Hitchin representation $\rho$ of a closed surface $S$, which we call the {\it limit curve} associated to $\rho$.  The image of this curve is a $C^{1+\al}$ submanifold and thus determines a measure class $\mu_\rho$ on $\bndry(S)$ via the pullback of the Lebesgue measure.  With respect to this measure class, we prove: 

\begin{restatable}{Thm}{thmmeasure}\label{thm:measure}
Let $S$ be an oriented closed surface and $\Sigma\subset S$ an incompressible subsurface.  Let $\rho$ be a Hitchin representation of $S$ and $\xi_\rho$ the associated limit curve.  If $\mu_\rho$ is the pullback of the Lebesgue measure on the image of $\xi_\rho$, then $\mu_\rho(\bndry(\Sigma)) = 0$. 
\end{restatable} 

This result generalizes a classical fact about the measure of the limit set of a subsurface of a closed hyperbolic surface (see \cite[Theorem 2.4.4]{Nicholls:1989ij}).

This paper is motived by and heavily relies on the framework introduced by Labourie and McShane in \cite{Labourie:2009gb} where the authors give an extension of the McShane-Mirzakhani identity  \cite{McShane:1991ug, McShane:1998dp, mirzakhani} to the setting of Hitchin representations.  The goal of \S\ref{sec:connections} is to relate these two identities. Both identities calculate the length of the boundary by giving full-measure decompositions.  We explain how these decompositions are related in both the classic hyperbolic setting as well as the Hitchin setting.  

In \S\ref{sec:projective} we give a geometric picture and motivation for some definitions and techniques by considering the case of 3-Hitchin representations, which correspond to convex real projective structures on surfaces as seen in the work of Choi and Goldman \cite{Goldman:1990vp, Choi:1993iz}. We also demonstrate that our formulation recovers Basmajian's identity for a Fuchsian representation.

\subsection*{Acknowledgements}
The first author would like to thank S.P. Tan for getting him to think about identities on real projective surfaces and would also like to thank T. Zhang for helpful conversations.  The second author would like to thank G. McShane for hosting him at the Insitut Fourier and for many helpful conversations. Both authors thank Andr\'es Sambarino for a helpful email correspondence and the referee for pointing out corrections and improving the exposition.

The first author was supported in part by NSF RTG grant 1045119. The second author acknowledges support from U.S. National Science Foundation grants DMS 1107452, 1107263, 1107367 "RNMS: GEometric structures And Representation varieties" (the GEAR Network).


\section{Background}\label{sec:background}

\subsection{Hitchin representations}\label{sec:hitchin}

Let $\Sigma$ be a connected compact oriented surface possibly with boundary and with negative Euler characteristic.  A homomorphism $\rho: \pi_1(\Sigma) \to \PSL(2,\br)$ is said to be \textit{Fuchsian} if it is faithful, discrete, and convex cocompact.  Let $\iota\co \PSL(2,\br) \to \PSL(n,\br)$ be a preferred representative  arising from the unique irreducible representation of $\SL(2,\R)$ into $\SL(n,\R)$. An \textit{$n$-Fuchsian} homomorphism is defined to be a homomorphism $\rho$ that factors as $\rho= \iota\circ \rho_0$, where $\rho_0$ is Fuchsian. 

Following the definition in \cite{Labourie:2009gb}, a \textit{Hitchin} homomorphism from $\pi_1(\Sigma)\to \PSL(n,\br)$ is one that may be deformed into an $n$-Fuchsian homomorphism such that the image of each boundary component stays purely loxodromic at each stage of the deformation. An element of $\PSL(n,\br)$ is \textit{purely loxodromic} if it has all real eigenvalues with multiplicity 1. 

For the rest of this paper, we will let $\rho$ denote the conjugacy class of a Hitchin homomorphism and refer to this class as a Hitchin representation.

\subsection{Doubling a Hitchin represenation}\label{doubling}
In this section, we will recall relevant details from the construction of Labourie and McShane on doubling Hitchin representations. See \cite[\S9]{Labourie:2009gb} for a complete discussion.

Let $\Sigma$ be a connected compact oriented surface with boundary whose double $\wh\Sigma$ has genus at least 2 and let $\rho$ be an $n$-Hitchin representation of $\pi_1(\Sigma)$. Fix a point $v \in \partial \Sigma$ and a primitive element $\partial_v \in \pi_1(\Sigma, v)$ corresponding to the boundary component containing $v$.  One can choose $R\co \pi_1(\Sigma, v) \to \mathrm{PSL}(n,\br)$ in the conjugacy class of $\rho$ such that $R(\partial_v)$ is a diagonal matrix with decreasing entries. Such a representative is called a \textit{good representative}.

There are two injections $\iota_0, \iota_1\co \Sigma \to \widehat \Sigma$ and an involution $\iota: \wh\Sigma \to \wh\Sigma$ fixing all points on $\partial \Sigma$ such that $\iota\circ\iota_0 = \iota_1$.   For $\gamma \in \pi_1(\wh \Sigma,v)$, define $\bar \gamma = \iota_*(\gamma)$.    
Let $J_n$ denote the $n\times n$ matrix whose $(i,j)^{th}$ entry satisfies:

\[
[J_n]_{ij} = 
\left\{
\begin{array}{l l}
(-1)^{i-1} & \text{when } i =j \\
0 & \text{otherwise}
\end{array}\right. 
\]
By \cite[Corollary 9.2.2.4]{Labourie:2009gb} there exists a unique Hitchin representation $\wh\rho$ of $\pi_1(\wh\Sigma)$ restricting to $\rho$ and satisfying the following condition. For any good representative $R$ of $\rho$ there exists $\wh R: \pi_1(\wh \Sigma,v) \to \mathrm{PSL}(n,\br)$ in the conjugacy class of $\wh\rho$ with
$$\wh R(\bar \gamma) = J_n\cdot\wh R(\gamma) \cdot J_n$$
for all $\gamma \in \pi_1(\wh \Sigma,v)$.  The representation $\wh \rho$ is called the \textit{Hitchin double of $\rho$} and we will refer to $\wh R$, as constructed from $R$, as a \textit{good representative of} $\wh\rho$.

From this construction and \cite[Theorem 1.5]{labourie1}, it follows that for a Hitchin representation $\rho$, the image $\rho(\gamma)$ of any nontrivial element of $\pi_1(\Sigma)$ is purely loxodromic.  In particular, associated to a Hitchin representation $\rho$ there is a length function $\ell_\rho$ defined by 
\begin{equation}\label{eq:length}
\ell_\rho(\gamma) := \log\left|\frac{\lambda_{\mathrm{max}}(\rho(\gamma))}{\lambda_{\mathrm{min}}(\rho(\gamma))}\right| \, ,
\end{equation}
where $\lambda_{\mathrm{max}}(\rho(\gamma))$ and $\lambda_{\mathrm{min}}(\rho(\gamma))$ are the eigenvalues of maximum and minimum absolute value of $\rho(\gamma)$, respectively.  Note that for a 2-Hitchin representation (i.e. a Fuchsian representation) this length function agrees with hyperbolic length.

\subsection{The boundary at infinity}\label{sec:boundary}

Let $\Sigma$ be a connected compact oriented surface  with negative Euler characteristic and choose a finite area hyperbolic metric $\sigma$ such that if $\partial \Sigma \neq \emptyset$, then $\partial\Sigma$ is totally geodesic.  We can identify the universal cover $\wt\Sigma$ of $\Sigma$ with $\bh^2$ if $\partial \Sigma = \emptyset$ or with a convex subset of $\bh^2$ cut out by disjoint geodesics in the case that $\partial \Sigma \neq \emptyset$.

One defines the \textit{boundary at infinity} $\bndry(\Sigma)$ of $\pi_1(\Sigma)$ to be $\overline{\wt \Sigma} \cap \partial_\infty \bh^2$. With this definition, it makes sense to talk about H\"older functions on $\bndry(\Sigma)$. Here, the metric on $\partial_\infty \bh^2$ comes from its identification with the unit circle $\Sp^1_\infty$. Recall that a map $f : X \to Y$ between metric spaces is $\al$-H\"older for $0 < \al \leq 1$, if there exists $C > 0$ such that, $$d_Y(f(x),f(y)) \leq C \,  d_X(x,y)^{\al} \quad \text{ for all }  x,y \in X$$
Clearly, H\"older functions are closed under composition, though the constant may change. For any two hyperbolic metrics $\sigma_1, \sigma_2$ on $\Sigma$, there exists a unique $\pi_1(\Sigma)$-equivariant quasisymmetric map $\bndry(\Sigma, \sigma_1) \to \bndry(\Sigma, \sigma_2)$ (see \cite[IV.A]{Ahlfors:1966ud}). This map is a H\"older homeomorphism (see \cite[Lemma 1]{Gardiner:2002em}) and therefore a H\"older map on $\bndry(\Sigma)$ will remain so if we choose a different metric. This definition of $\bndry(\Sigma)$ topologically coincides with the Gromov boundary of a hyperbolic group (see \cite[III.H.3]{Bridson:2013ky}), however the H\"older structure is additional.

Whenever $\Sigma$ is closed, $\bndry(\Sigma) \cong \Sp^1_\infty$. If $\Sigma$ has boundary and a double of genus at least 2, then $\bndry(\Sigma)$ is a Cantor set. Further, $\bndry(\Sigma)$ is identified as a subset of $\Sp^1_\infty$ and therefore admits a natural cyclic ordering from the orientation of $\Sigma$. For convention, we will view the ordering as counterclockwise.  

We will use the notation $(x, y) \subset \bndry(\Sigma)$ to denote the open set consisting of points $z$ such that the tuple $(x,z,y)$ is positively oriented.  Note that $(y,x) \cap  (x,y) = \emptyset$.

We say that a quadruple $(x,y,z,t)$ is {\it cyclically ordered} if the triples $(x,y,z)$, $(y,z,t)$ and $(z,t,x)$ are either all positively or negatively oriented.

\subsection{The Frenet Curve}\label{sec:frenet}

Let $\mathscr{F}$ be the complete flag variety for $\R^n$, i.e. the space of all maximal sequences $V_1 \subset V_2 \subset \cdots\subset V_{n-1}$ of proper linear subspaces of $\R^n$. Consider a curve $\Xi \co \Sp^1\to \mathscr{F}$ with $\Xi = (\xi_1, \xi_2, \ldots, \xi_{n-1})$. We say that $\Xi$ is a \textit{Frenet curve} if 
\begin{itemize}
\item for all sets of pairwise distinct points $(x_1, \ldots, x_l)$ in $\Sp^1$ and positive integers $d_1 +\cdots + d_l = d \leq n,$ 
$$ \bigoplus_{i=1}^l \xi_{d_i}(x_i) = \R^d\, .$$

\item for all $x$ in $\Sp^1$ and positive integers $d_1 +\cdots + d_l = d \leq n $,
$$\lim_{(y_1, \ldots, y_l)\to x, \atop y_i \text{ all distinct}} \left(\bigoplus_{i=1}^{i=l} \xi_{d_i}(y_i)\right) = \xi_d(x) \,.$$
\end{itemize}

We call $\xi = \xi_1$ the {\it limit curve} and $\theta = \xi_{n-1}$ the \textit{osculating hyperplane}. The second property above guarantees that the image of $\xi$ is a $C^{1+\al}$-submanifold of $\bp\br^n$. 

It turns out that given a Hitchin representation of a closed surface, one can construct an associated Frenet curve.  As a set of points, this curve is the closure of the attracting fixed points of $\rho(\gamma)$ for all $\gamma \in \pi_1(S)$.

\begin{Thm}[{\cite[Theorem 1.4]{labourie1}}]
\label{thm:frenet}
Let $\rho$ be an $n$-Hitchin representation of the fundamental group of a closed connected oriented surface $S$ of genus at least 2.  Then there exists a $\rho$-equivariant H\"older Frenet curve on $\bndry(S)$ .
\end{Thm}

The metric on $\mathscr{F}$ arises from a choice of inner product on $\br^n$ and the associated embedding $\mathscr{F} \to \prod_{i = 1}^{n-1} \bp\br^n$. In particular, we may use the usual spherical angle metric on $\im \xi_1$. Since $\xi_2$ is H\"older, we have the immediate corollary: 

\begin{Cor}\label{cor:c1a} If $\xi : \bndry(S) \to \mathscr{F}$ is the Frenet curve associated to an $n$-Hitchin representation, then $\im(\xi_1)$ is a $C^{1+\al}$ submanifold of $\bp\br^n$.
\end{Cor}

For a closed surface, let $\xi_\rho$ and $\theta_\rho$ be the limit curve and osculating hyperplane associated to a Hitchin representation $\rho$, respectively.  For a connected compact surface $\Sigma$ with boundary and a Hitchin representation $\rho$, we define $\xi_\rho$ to be the restriction of $\xi_{\wh\rho}$ to $\pi_1(\Sigma)$, where $\wh\rho$ is the Hitchin double of $\rho$.


\subsection{Cross Ratios}\label{sec:ratios}

The classical cross ratio of four points $x,y,z$ and $t$ in the projective line $\br\bp^1$ is defined as
\begin{equation}\label{eq:projective}
B_\bp(x,y,z,t) = \frac{|xy|\cdot |zt|}{|xt|\cdot|zy|} \, ,
\end{equation}
where distance is measured in any affine patch containing all four points.
Let $\Sigma$ be a connected compact oriented surface with double of genus at least 2.
Using the classical cross ratio above, we build a cross ratio associated to a Hitchin representation $\rho\co \pi_1(\Sigma) \to \PSL(n,\br)$ following \cite[Definition 4.2 and Proposition 5.8]{LabourieCross}.

For a projective line $L$ contained in $\bp\br^n$, let ${\bp\br_V^n}^* = \{ Z \in {\bp\br^n}^* : V \not\subset Z\}$ and let $\eta_V : {\bp\br_V^n}^* \to \bp\br^n$ be given by $\eta_V(w) = w \cap V$. For points $p,q \in \bp\br^n$ with $V = p \oplus q$ and  $r,s \in {\bp\br_V^n}^*$, define $$\mathfrak{B}(r,p,s,q) := B_V\left(\eta_V(r), p , \eta_V(s), q\right)\, ,$$
where $B_V$ is the classical cross ratio on $V$. Note that $\mathfrak{B}$ is a smooth function on its domain.

The \textit{cross ratio associated to $\rho$}, denoted $B_\rho$, is defined for $(x,y,z,t) \in \bndry(\Sigma)^{4*}$ by
\begin{equation}\label{eq:crossratio}
B_\rho(x,y,z,t) = \mathfrak{B}\left(\theta_\rho(x), \xi_\rho(y) , \theta_\rho(z), \xi_\rho(t)\right) \,,
\end{equation}
where
\[
\bndry(\Sigma)^{4*} = \{(x,y,z,t) \in \bndry(\Sigma)\, | \, x\neq t \text{ and } y\neq z\}\,.
\]
By construction, $B_\rho$ is H\"older and invariant under the diagonal action of $\pi_1(\Sigma)$.
Given $(x,y,z,t) \in \bndry(\Sigma)^{4*}$ such that both $(x,y,z)$ and $(x,t,z)$ are positively oriented, $B_\rho$ satisfies the following properties (see \cite[Section 4.2]{LabourieCross} for the closed case and \cite[Theorem 9.1]{labourie1} for the extenstion to the general compact setting):

\begin{align}
B(x,y,z,t) &= 0 \iff x=y \text{ or } z = t \, , \label{normalization1} \\
B(x,y,z,t) &= 1 \iff x = z \text{ or } y=t \, , \label{normalization2} \\
B(x,y,z,t) &= B(x,y,z,w)B(x,w,z,t)\, . \label{cocycle2} \\
B(x,y,z,t) &= B(x,t,z,y)^{-1}\, . \label{symmetry3} 
\end{align}
Furthermore, if $(x,t,y,z)$ is cyclically ordered, then 
\begin{equation}
\label{order1}
B(x,y,z,t) > 1.
\end{equation}

The \textit{period} of a nontrivial element $\gamma$ of $\pi_1(\Sigma)$ with respect to $B_\rho$, written $\ell_{B_\rho}(\gamma)$, is
\[
\ell_{B_\rho}(\gamma) = \log |B_\rho(\gamma^+, x, \gamma^-,  \gamma x)| =  \log |B_\rho(\gamma^-, \gamma x, \gamma^+, x)| \, ,
\]
where $\gamma^+$ (resp., $\gamma^-$) is the attracting (resp., repelling) fixed point of $\gamma$ on $\bndry(\Sigma)$ and $x$ is any element of $\bndry(\Sigma)\smallsetminus\{\gamma^+, \gamma^-\}$. This definition is independent of the choice of $x$ and satisfies
$$\ell_{B_\rho}(\gamma) = \ell_\rho(\gamma)$$
for any nontrivial element $\gamma$ of $\pi_1(\Sigma)$ \cite[Proposition 5.8]{LabourieCross}.

\textit{\textbf{Remark.}} We should note that the cross ratio associated to a Hitchin representation $\rho$ as defined here is referred to as $B_{\rho^*}$ in \cite{LabourieCross} and \cite{Labourie:2009gb}, where $\rho^*(\gamma) = \rho(\gamma^{-1})^t$. The cross ratio used in \cite{LabourieCross} and \cite{Labourie:2009gb} has $B_\rho(x,y,z,t) = B_{\rho^*}(y,x,t,z)$.  Both cross ratios have all the same properties, as shown in \cite{LabourieCross}.  The choice to use this definition is a cosmetic one for the case of $\br\bp^2$-surfaces considered below.


\section{Lebesgue Measure on the Frenet Curve}

Let $S$ be a closed surface and $\Sigma\subset S$ an incompressible connected subsurface.  A complete hyperbolic structure on $S$ gives an identification of $\bndry(S)$ with $\Sp^1_\infty = \partial_\infty \bh^2$.  It is a classical result that under this identification $\bndry(\Sigma)$ is measure 0 with respect to the Lebesgue measure on $\Sp^1_\infty$ (for instance, see \cite[Theorem 2.4.4]{Nicholls:1989ij}).  The goal of this section is to show that this holds true with respect to the Lebesgue measure on the limit curve associated to a Hitchin representation.  

For the entirety of this section, if $\rho$ is a Hitchin representation of a surface with boundary, we will use $R$ to denote a good representative. Further, we will assume that $\xi_\rho=\xi_R$. 

\begin{Lem}
Let $\wh \rho$ be the Hitchin double of $\rho\co \pi_1(\Sigma) \to \PSL(n,\br)$, then $J_n$ preserves the limit curve $\xi_{\wh\rho}\subset \bp\br^n$ associated to $\wh\rho$.
\end{Lem}

\begin{proof} Let $\xi = \xi_{\wh\rho} = \xi_{\wh R}$.  Since the attracting fixed points of $\wh R$ are dense in $\xi$, we will first show that $J_n$ preserves the set of attracting fixed points.

Let $\gamma \in \pi_1(\wh\Sigma)$, then by equivariance, $\xi(\gamma^+)$ is the attracting fixed point of $\wh R(\gamma)$. It follows that $J_n \cdot\xi(\gamma^+)$ is fixed by $J_n\cdot \wh R(\gamma)\cdot J_n = \wh R(\ov\gamma)$. Recall that $\ov\gamma$ is the image of $\gamma$ under the induced map of the canonical involution of $\wh\Sigma$.  Choose $x\notin \wh R(\ov\gamma)^\perp$ such that $y = J_n\cdot x \notin \wh R (\gamma)^\perp$. Here, $\wh R(\gamma)^\perp$ is the hyperplane spanned by the eigenvectors associated to the eigenvalues of non-maximal absolute value.  We then have that 
$$\lim_{k\to\infty} \left(\wh R(\bar\gamma)^k \cdot x\right) =\xi(\bar\gamma^+)$$
and also
\begin{align*}
\lim_{k\to\infty} \left(\wh R(\bar\gamma)^k \cdot x\right) & = \lim_{k\to\infty} \left(J_n\cdot \wh R(\gamma)^k \cdot J_n \cdot x\right) \\
&= J_n \cdot \left(\lim_{k\to\infty} \wh R(\gamma)^k \cdot y\right)\\ 
&= J_n \cdot \xi(\gamma^+)\, .
\end{align*}
In particular, $\xi(\bar\gamma^+) = J_n \cdot \xi(\gamma^+) \in \xi$ is the attracting fixed point of $\wh R(\bar \gamma)$.  Now choose $z\in \xi$, then there exists a sequence $\{\gamma_j\}$ in $\pi_1(\wh\Sigma)$ such that
$$\lim_{j\to\infty} \xi(\gamma_j^+) = z \,.$$
Hence, 
$$\lim_{j\to\infty} \left(J_n\cdot \xi(\gamma_j^+) \right)= J_n\cdot z$$
and as $\xi$ is closed, we have $J_n\cdot z \in \xi$.  Therefore $J_n$ preserves $\xi$. 
\end{proof}

Let $\xi_\rho \co \bndry(S) \to \bp \br^n$ be the limit curve associated to an $n$-Hitchin representation $\rho$ of $\pi_1(S)$.  By Corollary \ref{cor:c1a}, the image of $\xi_\rho$ is a $C^{1+\al}$ submanifold, so we let $\eta_\rho\co \Sp_\infty^1 \to \im(\xi)$ be a $C^{1}$-parameterization with $\al$-H\"older derivatives. We further assume that $\eta_\rho$ is constant speed with $\|\eta_\rho'\| = c_\rho > 0$ (recall that $\bp \br^n$ carries the standard spherical metric). Let $\lambda$ be the Lebesgue measure on $\Sp_\infty^1$ and define $$\mu_\rho = (\xi_\rho^{-1} \circ \eta_\rho)_*\lambda.$$

A finite positive measure $\mu$ on $\bndry(S)$ is \textit{quasi-invariant} if, for every $\gamma \in \pi_1(S)$, the pushforward measure $\gamma_* \mu$ is absolutely continuous with respect to $\mu$. In addition, if the Radon-Nikodym derivative is H\"older, we say $\mu$ is \textit{H\"older quasi-invariant}.

\begin{Lem}\label{lem:qi}
The measure $\mu_\rho$ is H\"older quasi-invariant.
\end{Lem}

\begin{proof}  Fix $\gamma \in \pi_1(S)$ and let $A \subset \bndry(S)$ be measurable. By definition, $$\gamma_* \mu_\rho(A) = \mu_\rho(\gamma^{-1}  A) = \lam\left(\eta_\rho^{-1} \circ \xi_\rho \left(\gamma^{-1}  A\right) \right) = \lam\left(\eta_\rho^{-1} \circ \rho(\gamma^{-1}) \circ \xi_\rho \left(A\right) \right).$$
Let $g = \rho(\gamma^{-1})$ and define $s_\gamma(t) = \eta_\rho^{-1} \circ g \circ \eta_\rho (t)$ then, \begin{equation}\label{eq:qi}\gamma_* \mu_\rho(A) = \lam(s_\gamma( \eta_\rho^{-1} \circ \xi (A))) = \int_{\eta_\rho^{-1} \circ \xi (A)} |s_\gamma'| \, d \lambda = \int_A |s_\gamma'| \circ \eta_\rho^{-1} \circ \xi_\rho  \, d\mu_\rho.\end{equation}
It remains to demonstrate that $|s_\gamma'| \circ \eta_\rho^{-1} \circ \xi_\rho $ is H\"older continuous.

We compute
\begin{align*}
\left\| \frac{d}{d t} \; \eta_\rho \circ s_\gamma(t) \right\| = \left\| \frac{d}{dt} \; g \circ \eta_\rho(t) \right\| \\
\\
c_\rho \cdot |s_\gamma'(t)| = \left\| D_{g}(\eta_\rho(t)) \cdot \eta_\rho'(t) \right\|
\end{align*}

Because $D_{g}$ is continuously differentiable on $\mathrm{T}\left(\bp\br^n\right)$ (and therefore H\"older) and $\eta_\rho'(t)$ is H\"older by construction, it follows that $|s_\gamma'(t)|$ is as well. 

For $\eta_\rho^{-1}$, consider the function $F$ on $\Sp_\infty^1\times\Sp_\infty^1$ given by $F(p,p) = 1/c_\rho$ and
\[
F(p,q) =  \frac{d_{\Sp^1}(p,q)}{d_{\bp\br^n}(\eta_\rho(p),\eta_\rho(q))} \text{ for } p \neq q.
\]
Since $\|\eta_\rho'\| = c_\rho > 0$, $F$ is continuous and therefore bounded.  
This shows that $\eta_\rho^{-1}$ is Lipschitz. 
Since $\xi_\rho$ is H\"older, it follows that $|s_\gamma'| \circ \eta_\rho^{-1} \circ \xi_\rho$ is H\"older and therefore $\mu$ is a H\"older quasi-invariant measure with respect to the action of $\pi_1(S)$ on $\bndry(S)$ by (\ref{eq:qi}).
\end{proof}

From Lemma \ref{lem:qi}, the work of Ledrappier \cite[III Section e]{Ledrappier:1994uy} tells us that that $\mu_\rho$ is $\pi_1(S)$-\textit{ergodic}. In particular, if $A\subset \bndry(S)$ is a $\pi_1(S)$-invariant set, then $A$ has either null or full measure. We now apply this property to obtain a special case of Theorem \ref{thm:measure}:

\begin{Lem}\label{lem:zero}
For a compact bordered surface $\Sigma$ with a double $\wh\Sigma$ of genus at least 2, fix $\rho$ a Hitchin representation of $\pi_1(\Sigma)$ and its  Hitchin double $\hat\rho$. Then, viewing $\bndry(\Sigma) \subset \bndry(\wh\Sigma)$, $$\mu_{\hat\rho}(\bndry(\Sigma)) = 0.$$ 
\end{Lem}

\begin{proof} Fix a basepoint in $\Sigma$ and consider $\bndry(\Sigma) \subset \bndry(\wh\Sigma)$ via the natural inclusion $\pi_1(\Sigma) \to \pi_1(\wh \Sigma)$. Let $\xi= \xi_{\wh \rho} = \xi_{\wh R}$.
Define
$$U = \bigcup_{g\in \pi_1(\wh\Sigma)} g\cdot\bndry(\Sigma)\, .$$
As $\mu_\rho$ is ergodic, either $\mu_\rho(U) = 0$ or $U$ has full measure.
Let $\iota$ be the involution on $\bndry ( \wh \Sigma )$ defined by $\xi^{-1}\circ J_n \circ \xi$. Then,
$$U' = \iota(U)$$
is another $\pi_1(S)$-invariant set implying it either has null or full measure.  Moreover, since $J_n|_{\im(\xi_\rho)}$ is $C^1$,
$$\mu_\rho(U) = 0 \iff \mu_\rho(U') = 0 \,.$$
Notice that both $\mu_\rho(U)$  and $\mu_\rho(U')$ cannot be full measure, as $U\cap U'$ consists of the attracting and repelling fixed points of primitive peripheral elements and must be countable. Therefore, $ 0 = \mu_\rho(U) \geq \mu_{\hat\rho}(\bndry(\Sigma))$\end{proof}
This measure property for the Hitchin double will be enough to prove the general case where $\Sigma \subset S$ is an incompressible surface. For this, we make use of the Hausdorff measure in $\br^{n-1}$.
\begin{Lem}[Theorem 3.2.3 \cite{Federer:1969iq}]\label{thm:haus} Let $f\co\br^m \to \br^n$ be a Lipschitz function for $m \leq n$. If $A$ is an $\lam^m$ (Lebesgue) measurable set, then 
$$\int_A \text{Jac}_m(f(x)) \, d \lam^m x = \int_{\br^n} N(f \mid A, y) \, d \mathcal{H}^m y$$
where $\mathcal{H}^m$ is the $m$-dimensional Hausdorff measure, $N(f \mid A, y)  = \# \{ x \in A \mid f(x)= y\}$, and $\text{Jac}_m(f(x)) = \sqrt{\det (Df^{t} \cdot Df)}(x)$.
\end{Lem}
\thmmeasure*
\begin{proof} By fixing a basepoint on $\Sigma$, there are natural inclusions $i\co\pi_1(\Sigma) \to \pi_1(S)$ and $\hat{\imath} \co \pi_1(\Sigma) \to \pi_1(\wh \Sigma)$ and the induced inclusions $i_*, \hat{\imath}_*$ on the boundaries at infinity. Let $R$ be a representative of $\rho$ such that $R \circ i$ is a good representative for $\left.\rho\right|_{\pi_1(\Sigma)}$ and build $\hat R$ by doubling $R \circ i$. 

Let $\xi = \xi_R$ and $\hat \xi = \xi_{\wh R}$ be the limit curves associated to $\rho$ and $\hat \rho$, respectively. For $\gamma\in \pi_1(\Sigma)$, 
$$\xi \circ i_*(\gamma^+) = R(i(\gamma))^+  = \wh R(\hat{\imath}(\gamma))^+ = \hat\xi \circ \hat{\imath}_*(\gamma^+)$$
as $R(i(\gamma)) = \wh R(\hat{\imath}(\gamma))$. By the density of attracting fixed points in $\bndry(\Sigma)$ we see that $\xi \circ i_* = \hat\xi \circ \hat{\imath}_*$. In particular, they have the same image 
\[
\Lambda_\Sigma = \xi \circ i_*(\bndry(\Sigma)) =   \hat\xi \circ \hat{\imath}_*(\bndry(\Sigma)).
\]

Let $\eta_\rho, \eta_{\hat\rho} \co \Sp_\infty^1 \to \bp\br^n$ be the two $C^{1+\al}$ constant speed parametrizations of $\im(\xi)$ and $\im(\hat\xi)$.
Fix a cover of $\bp\br^n$ by affine charts $\{U_1, \ldots, U_k\}$ and for $i \in \{1, \ldots, k\}$ let $\mathfrak{p}_i \co U_i \to \br^{n-1}$ be the coordinate functions.
Applying Lemma \ref{thm:haus} to $\mathfrak{p}_i \circ \eta_\rho$ we obtain: 

\[
\mu_{\rho}(i_*(\bndry(\Sigma))) \leq \sum_{i=1}^k\int_{\eta_{\rho}^{-1}(U_i\cap\Lambda_\Sigma)} \text{Jac}_1(\mathfrak{p}_i \circ \eta_\rho) \; d \lam = \sum_{i=1}^k \mathcal{H}^1(\mathfrak{p}_i(U_i\cap\Lambda_\Sigma)).
\]

An application of Lemma \ref{thm:haus} and Lemma  \ref{lem:zero} to $\mathfrak{p}_i \circ \eta_{\hat{\rho}}$ yields:
\[
\mathcal{H}^1(U_i\cap\mathfrak{p}_i (\Lambda_\Sigma)) = \int_{\eta_{\hat\rho}^{-1}(U_i\cap\Lambda_\Sigma)} \text{Jac}_1(\mathfrak{p}_i \circ \eta_{\hat\rho}) \; d \lam \leq \mu_{\hat{\rho}}(\hat{\imath}_*(\bndry(\Sigma))) = 0.
\]
Therefore, $\mu_{\rho}(i_*(\bndry(\Sigma)))  = 0$ as desired.
\end{proof}


\section{Orthogeodesics and Double Cosets}\label{sec:cosets}

Let $\Sigma$ be a connected compact orientable surface with genus $g$ and $m>0$ boundary components such that the double of $\Sigma$ has genus at least 2.  Fix a finite volume hyperbolic metric $\sigma$ on $\Sigma$ such that $\partial \Sigma$ is totally geodesic.   In particular, we can fix an identification of the universal cover $U$ of $\Sigma$ with a convex subset of $\bh^2$ cutout by geodesics.  This also gives an identification of $\pi_1(\Sigma)$ with a discrete subgroup of $\mathrm{Isom}^+(\bh^2)$.  

An \emph{orthogeodesic} in $(\Sigma, \sigma)$ is an oriented properly embedded arc perpendicular to $\partial\Sigma$ at both endpoints.  Denote the collection of orthogeodesics as $\cO(\Sigma,\sigma)$.   The \emph{orthospectrum} is the multiset containing the lengths of orthogeodesics with multiplicity and is denoted by $|\cO(\Sigma,\sigma)|$.  Observe that every element of  $|\cO(\Sigma,\sigma)|$ appears at least twice as orthogeodesics are oriented.  Also note that $\cO(\Sigma,\sigma)$ is countable as the orthogeodesics  correspond to a subset of the oriented closed geodesics in the double of $(\Sigma, \sigma)$.  Let $\ell_\sigma(\partial \Sigma)$ be the length of $\partial \Sigma$ in $(\Sigma,\sigma)$, then recall Basmajian's identity \cite{Basmajian:1993eu}
$$\ell_\sigma(\partial \Sigma) = \sum_{\ell \in |\cO(\Sigma,\sigma)|} 2 \log \coth \left(\frac \ell2 \right) \,.$$

In order to extend this identity to the setting of Hitchin representations, we first need to replace the geometric object $\cO(\Sigma,\sigma)$ with an algebraic object; this is the goal of this section.

Let $\cA=\{\al_1, \ldots, \al_m\}\subset \pi_1(\Sigma)$ be a collection of primitive elements representing  the $m$ components of $\partial \Sigma$ in $\pi_1(\Sigma)$ oriented such that the surface is to the left. We will call such a set $\cA$ a \textit{positive peripheral marking}. Set $H_i = \la \al_i \ra$ and, treating $\pi_1(\Sigma)$ as a subgroup of $\PSL(2,\R)$, let $\wt \al_i\subset \bh^2$ be the lift of $\al_i$ such that $H_i = \mathrm{Stab}(\wt\al_i)$. 

Fix $1\leq i, j\leq n$ (possibly $i=j$), then for $g\in \pi_1(\Sigma)$ (or $g\in \pi_1(\Sigma) \smallsetminus H_i$ if $i=j$) define the arc $\wt\al_{i,j}(g)$ to be the minimal length arc oriented from $\wt\al_i$ to $g\cdot \wt \al_j$.  Now $\wt \al_{i,j}(g)$ descends to an orthogeodesic $\al_{i,j}(g)$ on $(\Sigma, \sigma)$.  For $i \neq j$, we denote the set of double cosets
$$\cO_{i,j}(\Sigma, \cA) = H_i\bs \pi_1(\Sigma) / H_j = \{H_i g H_j \colon g\in \pi_1(\Sigma)\}$$
and for $i = j$, define
$$\cO_{i,i}(\Sigma, \cA) = \left(H_i\bs \pi_1(\Sigma) / H_i\right) \smallsetminus \{H_i e H_i\}\,,$$
where $e\in \pi_1(\Sigma)$ is the identity. We will denote an element of $\cO_{i,j}(\Sigma, \cA)$ corresponding to $H_i g H_j$ as $[g]_{i,j}$. Associated to the pair $(\Sigma,\cA)$ we define the \emph{orthoset} to be the collection of all such cosets
$$\cO(\Sigma, \cA) = \bigsqcup_{1\leq i,j \leq n} \cO_{i,j}(\Sigma, \cA) \, .$$

From the definitions, it is clear that the map 
$$\Phi: \cO(\Sigma, \cA) \to \cO(\Sigma,\sigma)$$
given by 
$$\Phi([g]_{i,j}) = \al_{i,j}(g)$$
is well-defined.  

\begin{Prop}\label{lem:bijection}
The map $\Phi$ is a bijection.
\end{Prop}

\begin{proof}
We first show it is injective.  Suppose $\Phi([g]_{i,j}) = \Phi([g']_{i',j'})$.  First note that $i'=i$ and $j=j'$ since the arcs must be oriented from $\al_i$ to $\al_j$.  Now $\wt \al_{i,j}(g)$ and $\wt \al_{i,j}(g')$ must differ by an element of $\pi_1(\Sigma)$.  Since both these arcs start on $\wt\al_i$ it is clear that there exists $h_i \in H_i$ such that 
$$\wt \al_{i,j}(g) = h_i \cdot \wt\al_{i,j}(g')\,.$$
In particular, we must have that $g \cdot \wt\al_j = (h_i g')\cdot \wt\al_j$ implying 
$$ (g')^{-1} h_i^{-1} g \in H_j \,.$$
Set $h_j = (g')^{-1}h_i^{-1} g \in H_j$, then 
$$g = h_i g' h_j \in H_i g' H_j \,,$$
so that $[g]_{i,j} = [g']_{i,j}$ and $\Phi$ is injective.  

To see that $\Phi$ is surjective, take an orthogeodesic $\be \in \cO(\Sigma, \sigma)$ from $\al_i$ to $\al_j$.  Choose a lift $\wt\be$ of $\be$ such that $\wt\be$ starts on $\wt\al_i$.  But $\wt\be$ must also end on some lift of $\al_j$ which we can write as $g\cdot \wt\al_j$, so that $\Phi([g]_{i,j}) = \be$ and $\Phi$ is surjective. Notice that if $i = j$, then $g \notin H_i e H_i$ as $\be$ is a non-trivial orthogeodesic.
\end{proof}

We will see how to rewrite Basmajian's identity in terms of the orthoset as a corollary of generalizing the identity to real projective structures.  
\begin{Rem}(1) In his paper, Basmajian \cite{Basmajian:1993eu} uses the fact that an orthogeodesic can be obtained from $g\in \pi_1(\Sigma)$; in our notation, he constructs $\al_{i,i}(g)$ for a fixed $i$.\\
(2) Despite using the language and setting of surfaces, the above discussion holds just as well for connected compact hyperbolic $n$-manifolds with totally geodesic boundary.
\end{Rem}


\section{Real Projective Structures $(n = 3)$}\label{sec:projective}

A convex real projective surface, or convex $\rp^2$-surface, is a quotient $\Omega/ \Gamma$ where $\Omega\subset \rp^2$ is a convex domain in the complement of some $\rp^1$ and $\Gamma < \mathrm{PGL}(3,\br)$ is a discrete group acting properly on $\Omega$.  A convex $\rp^2$-structure on a surface $S$ is a diffeomorphism $f: S \to \Omega/ \Gamma$.
When $S$ has negative Euler characteristic, the work of Goldman \cite{Goldman:1990vp} tells us that the conjugacy class of the holonomy coming from a convex $\rp^2$-structure on a surface $S$ is a Hitchin representation  $\pi_1(S) \to \mathrm{PSL}(3,\br)$.  In fact, for closed surfaces this identification is a bijection by the work of Choi-Goldman \cite{Choi:1993iz}.

In this section we give a generalization of Basmajian's identity to convex $\rp^2$-surfaces and by extension to 3-Hitchin representations.  This result is an immediate corollary of Theorem \ref{thm:identity}, however, the proof here is geometric in nature and will closely follow Basmajian's original proof in \cite{Basmajian:1993eu}.  Further, it motivates the general case.

\subsection{Hilbert metric}

Let $F = \Omega/ \Gamma$ be a convex $\rp^2$-surface with negative Euler characteristic, then $F$ carries a natural Finsler metric called the Hilbert metric, which we now describe.

Let $x,y \in \Omega$ and define $L\subset \rp^2$ to be the projective line connecting $x$ and $y$.  $L$ intersects $\partial \Omega$ in two points $p,q$ such that $p,x,y,q$ is cyclically ordered on $L$.  Choose any affine patch containing these four points, then the Hilbert distance between $x$ and $y$ is 
$$h(x,y) := \log B_\bp(p,y,q,x) \, ,$$
where $B_\bp$ is the projective cross-ratio defined in \eqref{eq:projective}.  The geodesics in the Hilbert geometry correspond to the intersection of projective lines with $\Omega$.  As the cross-ratio is invariant under projective transformations we see that the Hilbert metric descends to a metric on $F$.

Let $\rho$ be the holonomy associated to a convex $\rp^2$-structure on a surface $S$, then for a primitive element $g\in \pi_1(\Sigma),$ the length $\ell_\rho(g)$ (see \eqref{eq:length}) agrees with the translation length of the geodesic representative of $\rho(g)$ in the Hilbert metric.  

Note that when $\Omega$ is a conic, it is projectively equivalent to a disk. In this case, $F$ is hyperbolic and $h = 2d_\bh$ where $d_\bh$ denotes the hyperbolic metric.
For more details on Hilbert geometry see \cite{Busemann:2012vp}.

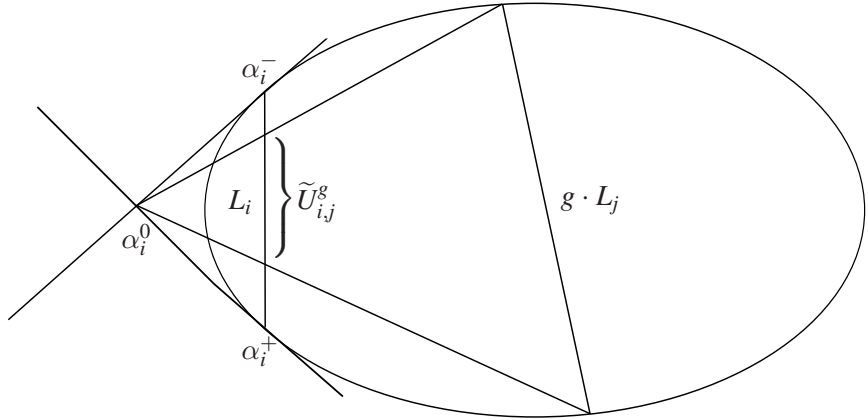
\begin{figure}[t]
\begin{center}\begin{overpic}[scale=0.7]{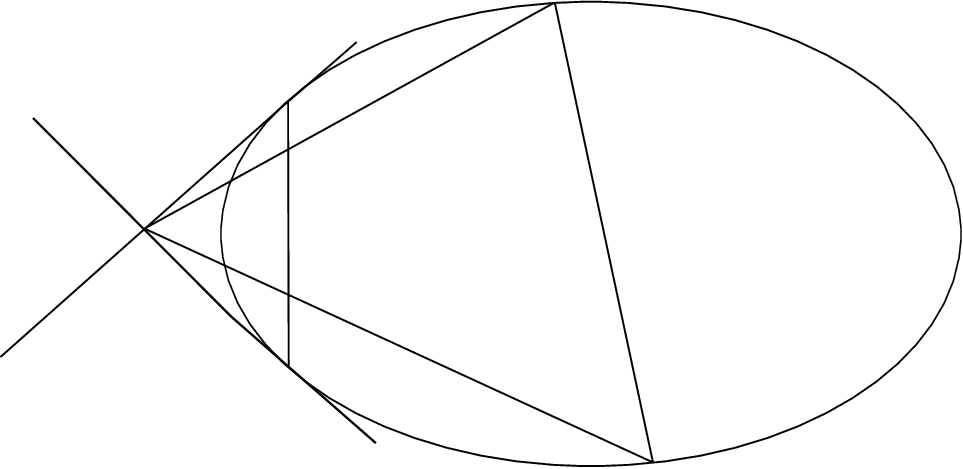}
\put(30,25){$\left.\rule{0pt}{.95cm}\right\}$}
\put(33.3,24.5){${\wt U^g_{i,j}}$}
\put(64.2,25){$g\cdot L_j$}
\put(25.5,24.5){${L_i}$}
\put(13,20){$\al_i^0$}
\put(27,7){$\al_i^+$}
\put(27,40){$\al_i^-$}
\end{overpic}
\end{center}
\caption{Orthogonal projection of $g\cdot L_j$ onto $L_i$ whose image we defined as $\wt U^g_{i,j}$.}
\label{fig:projection}
\end{figure}

\subsection{Basmajian's identity}
Let $F$ be a connected compact orientable convex $\rp^2$-surface with non-empty totally geodesic boundary whose double is at least genus 2. 
Using the doubling construction described in \S\ref{doubling} for Hitchin representations, let $\wh F = \Omega/ \Gamma$ be the double of $F$. Then $\wh F$ is a closed convex $\rp^2$-surface. Note that there is also a doubling construction in \cite{Goldman:1990vp} inherent to convex $\br\bp^2$-surfaces, which is essentially a more geometric version of the Hitchin doubling that we have already discussed.

Choose a positive peripheral marking $\cA = \{\al_i\}_{i=1}^m$.  
Let $\wt F \subset \Omega$ be the universal cover of $F$ and let $L_i$ be the geodesic in $\Omega$ stabilized by $\al_i \in \Gamma$.  
In projective geometry, orthogonal projection to $L_i$ is defined as follows: as the boundary of $\Omega$ is $C^1$ \cite{Goldman:1990vp}, for $x \in \partial \Omega$ let $\theta(x)$ denote the line tangent to $\partial \Omega$ at $x$.  Set $\al^0_i = \theta(\al_i^+) \cap \theta(\al_i^-)$, then the projection to $L_i$ is defined to be $\eta_i: \overline \Omega \to L_i$ where $\eta_i(y)$ is the intersection of the line connecting $\al^0_i$ and $y$ and the line $L_i$. 
For $[g]_{i,j} \in \cO(F, \cA)$, we let 
$$\wt U_{i,j}^g = \eta_i(g\cdot L_j)$$
be the orthogonal projection of $g\cdot L_j$ onto $L_i$.  This is shown in Figure \ref{fig:projection}.  

\begin{Lem}\label{lem:injective}
Let $\pi: \Omega \to \wh F$ be the universal covering map, then $\pi|\wt U_{i,j}^g$ is injective.
\end{Lem}

\begin{proof}
Suppose that $\pi|\wt U_{i,j}^g$ were not injective, then $(\al_i \cdot U_{i,j}^g )\cap U_{i,j}^g \neq 0$.  This can only happen if $(\al_ig)\cdot L_j$ and $g\cdot L_j$ intersect in $\Omega$, which is impossible as the boundary is totally geodesic. 
\end{proof}

By Lemma \ref{lem:injective}, we may define $U_{i,j}^g = \pi(\wt U_{i,j}^g)$.

\begin{Lem}
If  $[g]_{i,j}, [h]_{r,s} \in \cO(F, \cA)$ are distinct elements, then $U_{i,j}^g \cap U_{r,s}^h = \emptyset$.
\end{Lem}

\begin{proof}
If $U_{i,j}^g$ intersects $U_{r,s}^h$, then $i = r$ and by fixing lifts, one has $g \cdot L_j \cap h \cdot L_s \neq\emptyset$, which would mean that $\partial F$ is not a totally geodesic 1-submanifold.
\end{proof}

We define $G_F\co \cO(F, \cA) \to \br^+$ by
$$G_F([g]_{i,j}) = \log B_\bp(\al_i^+, \eta_i(g\cdot \al_j^+), \al_i^-, \eta_i(g\cdot \al_j^-))$$
for $[g]_{i,j} \in \cO(F, \cA)$.  Let $\rho$ be a $3$-Hitchin representation realizing $F$, then by a standard fact in projective geometry about cross-ratios of four lines $$G_F([g]_{i,j}) = \log B_\rho(\al_i^+, g\cdot \al_j^+, \al_i^-, g\cdot \al_j^-)\, ,$$
which agrees with our function in Theorem \ref{thm:identity}. We can then write Basmajian's identity:

\begin{Prop}[Basmajian's identity for $\rp^2$-surfaces]\label{prop:rp2}
Let $F$ be a connected compact orientable convex $\rp^2$-surface with non-empty totally geodesic boundary whose double has genus at least 2. Let $\cA=\{\al_1, \ldots, \al_m\}$ be a positive peripheral marking.  Then,
$$\ell_F(\partial F) = \sum_{x\in \cO(F,\cA)} G_F(x)$$
where $\ell_F$ measures length in the Hilbert metric on $F$ and $\ell_F(\partial F) = \sum_{i=1}^n \ell_F(\al_i)$.  Furthermore, if $F$ is hyperbolic, then this is Basmajian's identity.  
\end{Prop}

\begin{proof}
Abusing notation, we will use $\al_i$ to denote both the element in $\pi_1(F)$ and its geodesic representative in $F$.  From above, we have $U_{i,j}^g$ is an interval embedded in $\al_i$ and by construction
$$\ell(U_{i,j}^g) = \log B_\bp(\al_i^+, \eta_i(g\cdot \al_j^+), \al_i^-, \eta_i(g\cdot \al_j^-)) \,.$$
For a fixed $i$, the complement of 
$$\bigcup_{[g]_{i,j} \in \mathscr{O}(F,\cA)} U_{i,j}^g$$
in $\al_i$ is the projection of $\bndry(F)$, or $\pi(\eta_i(\bndry(F)))$, which has measure zero by Lemma \ref{lem:zero} and the fact that $\pi \circ \eta_i$ is differentiable. This gives the identity as stated.

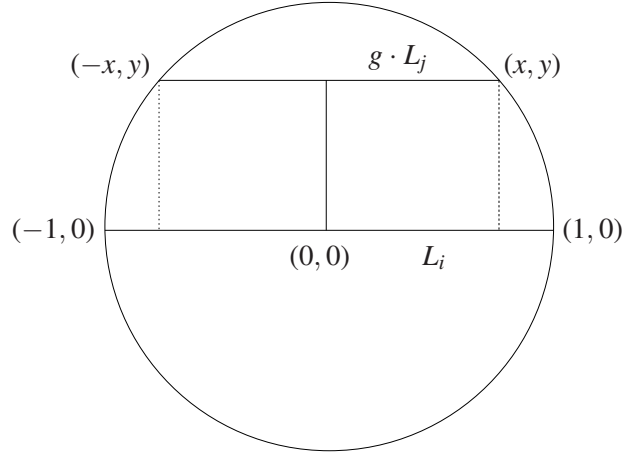
\begin{figure}[t]
\begin{center}\begin{overpic}[scale=0.35]{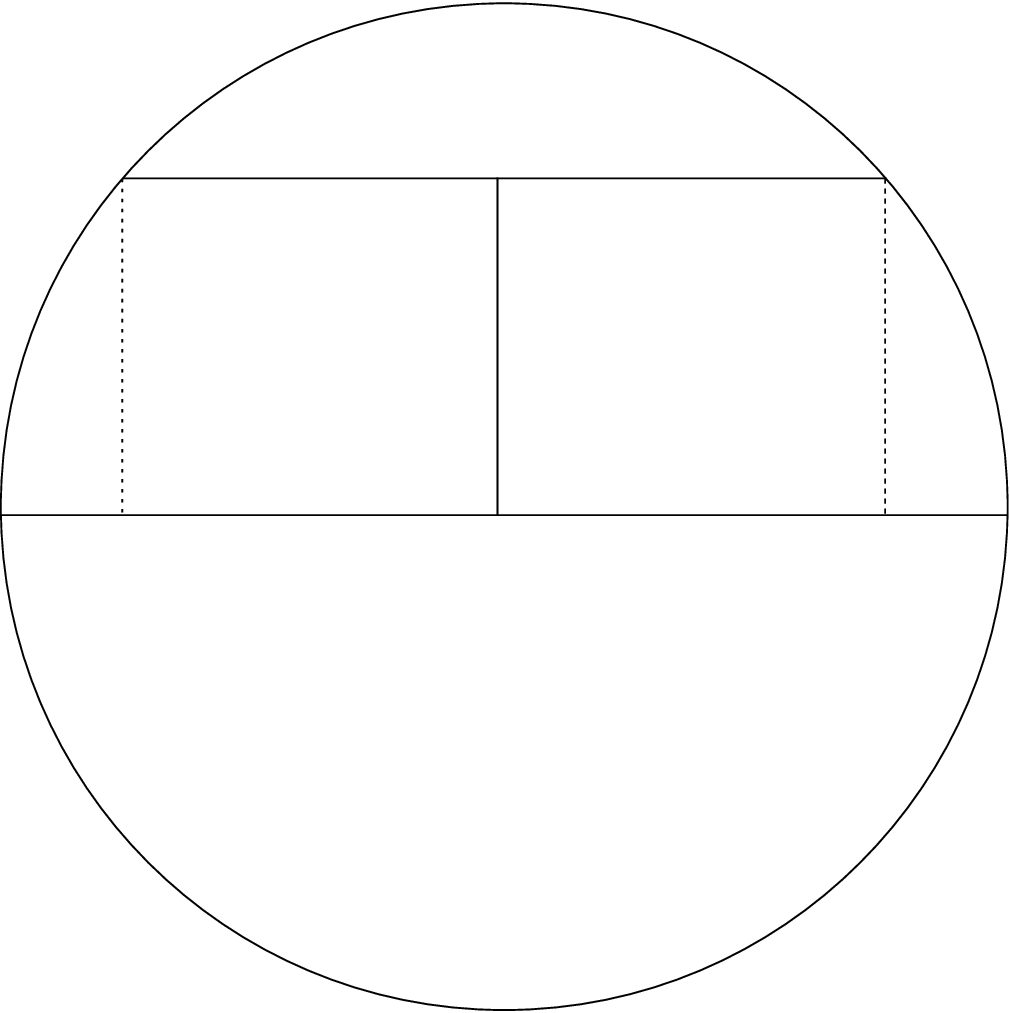}
\put(88,84){$(x,y)$}
\put(-8,84){$(-x,y)$}
\put(40.5,41.5){$(0,0)$}
\put(-21,47.5){$(-1,0)$}
\put(101,47.5){$(1,0)$}
\put(70,41.5){$L_i$}
\put(58,86){$g\cdot L_j$}
\end{overpic}
\end{center}
\caption{A standard diagram for the orthogonal projection of $g\cdot L_j$ onto $L_i$ in the hyperbolic case.}
\label{fig:hyperbolic}
\end{figure}

We now show that this is Basmajian's identity in the case that $\Sigma$ is hyperbolic. 
In this case, we may draw a standard picture with $\Omega$ being the unit disk in an affine patch as in Figure \ref{fig:hyperbolic}.  The line connecting $g\cdot L_j$ and $L_i$ is a lift of the orthogeodesic corresponding to the element $[g]_{i,j}$ (this can be seen by considering the corresponding geodesics in the Poincar\'e disk model).  We have
$$\ell = \log B_\bp((0,-1), (0,y), (0,1), (0,0)) = \log\left( \frac{1+y}{1-y}\right)$$
is the length of this orthogeodesic in the Hilbert metric and let
$$L = \log B_\bp((-1,0), (x,0), (1,0), (-x,0))= 2\log \left(\frac{1+x}{1-x}\right)$$
be the length of the projection of $g\cdot L_j$ onto $L_i$.  From this we see that 
$$x = \tanh \frac{L}4 \text{ and } y = \tanh \frac{\ell}2 \,.$$
From $x^2+y^2 =1$ we see that 
$$1 = \tanh^2\left(\frac{L}{4}\right) + \tanh^2 \left(\frac{\ell}{2}\right) \quad \Lra \quad \tanh^2\left(\frac{L}{4}\right) = \sech^2 \left(\frac{\ell}{2}\right) \,.$$
Now using the fact that 
$$\mathrm{arctanh}(z) = \frac12 \log\left( \frac{1+z}{1-z}\right)$$
we have
$$L = 4 \text{arctanh}\left( \sech \left(\frac{\ell}{2}\right)\right) = 2 \log\left(\frac{1+  \sech \left(\frac{\ell}{2}\right)}{1- \sech \left(\frac{\ell}{2}\right)}\right) = 4 \log \coth\left( \frac{\ell}4\right).$$
Recalling that the Hilbert metric is twice the hyperbolic metric, we recover
$$\ell_h(\partial F) = \sum_{\ell_h\in |\cO(F)|}2 \log\coth\left(\frac{\ell_h}{2}\right) \, ,$$
where $\ell_h(\gamma)$ measures length of $\gamma$ in the hyperbolic metric on $F$, which is as desired.
\end{proof}


\section{Basmajian's Identity}

We saw in the case of convex $\rp^2$-structures (or $3$-Hitchin representations) on a bordered surface that Basmajian's identity is derived by computing the lengths of orthogonal projections in the universal cover.  In the $n$-Hitchin case, we no longer have the same picture of a universal cover (for $n>3$), but the idea is roughly the same.  In fact, in terms of cross ratios, we will be using the same function on the orthoset as the summand.  

Let $\Sigma$ be a compact surface with $m>0$ boundary components whose double has genus at least 2.  Choose a positive peripheral marking $\cA = \{\al_1, \cdots, \al_m\}$, then for a Hitchin representation $\rho$ we define the function $G_\rho \co \cO(\Sigma, \cA) \to \br_+$ by 
\[
G_\rho\left([g]_{i,j}\right) = \log B_\rho \, (\al_i^+, g\cdot \al_j^+, \al_i^-, g\cdot \al_j^-)\,.
\]
We think of $G_\rho([g]_{i,j})$ as measuring the length of the projection of the 
line connecting $g\cdot \al_j^+$ and $g\cdot \al_j^-$ to the line connecting $\al_i^+$ and $\al_i^-$.  

\basidentity*

Before we start the proof, we need a short lemma.  For a surface $\Sigma$ (with or without boundary) recall that $\bndry(\Sigma)$ has an ordering and for $x,y \in \bndry(\Sigma)$, we defined
\[
(x,y) = \{ z \in \bndry(\Sigma) : (x,z,y) \text{ is positive}\}.
\]
Given an element $\al\in \pi_1(\Sigma)$ and a reference point $\zeta \in (\al^+, \al^-) \subset \bndry(\Sigma)$, define the function $F_\rho\co (\al^+, \al^-)\to \br$  by
\begin{equation}\label{eq:projection}
F_\rho(x) = \log B_\rho(\al^+, x, \al^-, \zeta) \, .
\end{equation}
Note that  $B_\rho(\al^+, x, \al^-, \zeta)$ is positive by \eqref{order1} and \eqref{symmetry3}.

\begin{Lem}\label{L: homeo}
$F_\rho$ is a homeomorphism onto its image.  Further, if $\Sigma$ is closed, then $F_\rho$ is surjective.
\end{Lem}

\begin{proof} A proof of this fact is given in the proof of \cite[Theorem 4.1.2.1]{Labourie:2009gb}. We include the argument here for completeness.  First injectivity:  if $B_\rho(\al^+, x,\al^-, \zeta) = B_\rho(\al^+, x', \al^-, \zeta)$, then $B_\rho(\al^+, x, \al^-, x') = 1$  by \eqref{cocycle2}; hence, $x=x'$ by \eqref{normalization2}. Furthermore, the inequality \eqref{order1} implies $F_\rho$ preserves the ordering and therefore it is a homeomorphism onto its image.  Lastly, note that as $x\to \al^\pm$ we have $F_\rho(x) \to \mp\infty$ by \eqref{normalization1} and \eqref{symmetry3} and that $(\al^+, \al^-)$ is connected if $\Sigma$ is closed.
\end{proof}

\begin{proof}[Proof of Theorem \ref{thm:identity}] 
We use the framework from \cite[Theorem 4.1.2.1]{Labourie:2009gb}. Let us focus our attention on a single boundary component.  Let $\al = \al_1$.  Fix a finite area hyperbolic structure on $\Sigma$ so that $\partial \Sigma$ is totally geodesic. Identify $\Sigma$ with $U/\Gamma$ for a convex set $U\subset \bh^2$ whose boundary in $\bh^2$ is a disjoint union of geodesics. With this identification, $\bndry (\Sigma) \iso \partial_\infty U = \overline{U}\cap \Sp^1_\infty$ and $\pi_1(\Sigma) \iso \Gamma$.  Moreover, $\Sp^1_\infty\smallsetminus \partial_\infty U$ is a union of {\it disjoint} intervals of the form $\tilde I_\be = (\be^-, \be^+)$ for primitive peripheral elements $\be \in \Gamma$ which have $\Sigma$ on their left. By construction, $\beta =  g \al_j g^{-1}$ for some $a_j$ in the positive peripheral marking $\cA$ and $g \in \Gamma$.

Observe that $\left(g \al_j g^{-1}\right)^\pm = \al_k^\pm$ if and only if $g \al_j g^{-1} = \al_k$ because $g \al_j g^{-1}$ and $\al_k$ are primitive. In particular, we must have that $j = k$ and $g \in H_j = \langle \al_j \rangle$. We therefore conclude that $\tilde I_{g_1 \al_j g_1^{-1}} = \tilde I_{g_2 \al_k g_2^{-1}}$ if and only if $j = k$ and $g_2^{-1}g_1 \in H_j$; hence, for $\be = g \al_j g^{-1}$, the map
\[
\tilde I_\beta \mapsto gH_j,
\]
gives a bijection
$$ \left\{ \text{Components } \tilde I_\be  \text{ of  } \Sp^1_\infty \smallsetminus \partial_\infty U\right\} \Longleftrightarrow \bigsqcup_{1\leq j \leq n} \pi_1(\Sigma)/H_j \, .$$

Let $B = B_\rho$ be the cross ratio associated to $\rho$ and fix some $\zeta \in (\al^+, \al^-) \subset \bndry(\Sigma)$ in order to define $F_\rho\co (\al^+, \al^-)\to \br$ as above in Lemma \ref{L: homeo}. 
Since $F_\rho$ is increasing, we see that the set $\br \smallsetminus F_\rho(\partial_\infty U)$ is a union of disjoint intervals $\hat I_\be = (F_\rho(\be^-), F_\rho(\be^+))$. Further, \begin{align*}
F_\rho(\al \cdot x)& =  \log B(\al^+, \al \cdot x, \al^-, \zeta) \\
&  = \log \frac{B(\al^+, x, \al^-, \zeta)}{B(\al^+, x, \al^-, \al \cdot x)}\\
& = F_\rho(x)-\ell_\rho(\al)
\end{align*}
by \eqref{cocycle2} and \eqref{symmetry3}. Now, set $\mathbb{T} = \br /\ell_\rho(\al) \bz$ and define $\pi: \br \to \mathbb{T}$ to be the projection.  
From above, we have that  $\hat I_{\al \be \al^{-1}} \cap \hat I_ \be = \emptyset$ and $$\hat I_{\al \be \al^{-1}}  = (F_\rho(\al \cdot  \be^+), F_\rho(\al \cdot \be^-)) = \hat I_ \be - \ell_\rho(\al)$$ so $\pi|_{\hat I_\be}$ is injective.  Define $I_\be = \pi(\hat I_\be)$, then we have the bijection
$$ \left\{ \text{Components } I_\be  \text{ of  } \mathbb{T} \smallsetminus \pi(F_\rho(\partial_\infty U)) \right\} \Longleftrightarrow \left(\bigsqcup_{1\leq j \leq m} H_1 \bs \pi_1(\Sigma)/H_j \right) \smallsetminus \{H_1 e H_1\} \, ,$$
where we remove $H_1 e H_1$ as it corresponds to the interval $\tilde I_\al$, which is outside $(\al^+, \al^-)$. Using our notation from \S\ref{sec:cosets}, the right hand side is simply $\bigsqcup_{1\leq j \leq m} \cO_{1,j}(\Sigma, \cA)$.

For each $I_\be$, there is a $j$ and an element $[g]_{1,j} \in \cO_{1,j}(\Sigma, \cA)$, where $\be = g\al_jg^{-1}$.  With this representative, we see that if $\lambda$ is the Lebesgue measure on $\br$, then
\begin{align*}
\lambda(I_\be)
&= F_\rho(\be^+) - F_\rho(\be^-) \\
&= \log \frac{B(\al^+, \be^+, \al^-, \zeta)}{B(\al^+, \be^-, \al^-, \zeta)} \\
&= \log\left(B(\al^+, \be^+, \al^-, \zeta)\cdot B(\al^+, \zeta, \al^-, \be^-)\right) \text{ (by \eqref{symmetry3})} \\
&= \log B(\al^+, \be^+, \al^-, \be^-) \text{ (by \eqref{cocycle2})}\\
&= \log B(\al^+, g\cdot \al_j^+, \al^-, g \cdot \al_j^-)\\
& = G_\rho([g]_{1,j}) \, .
\end{align*}
 
It follows that \begin{equation}\label{eq:onecpt}\ell_\rho(\al) = \lambda(\mathbb{T}) =  \lambda(\pi(F_\rho(\partial_\infty U)))\, \,  + \sum_{1 \leq j \leq m} \; \sum_{x \in \cO_{1,j}(\Sigma, \cA)} G_\rho(x) \, .\end{equation}
By definition, $\pi\circ F_\rho(x) =  \pi \circ \log \circ \mathfrak{B}\left(\theta_\rho(\al^+), \xi_\rho(x), \theta_\rho(\al^-), \xi_\rho(\zeta)\right)$. Thus, Lemma \ref{lem:zero} and the fact that $\pi \circ \log \circ \mathfrak{B}$ is differentiable, tells us that $\lambda(\pi(F_\rho(\partial_\infty U))) = 0$. This gives the identity for a single boundary component. By doing the same for the other boundary components and summing, we have arrived at 
$$\ell_\rho(\partial \Sigma) = \sum_{x\in \cO(\Sigma, \cA)} G_\rho(x) \,.$$
We finish by noting that the proof of Proposition \ref{prop:rp2} implies that if $\rho$ is Fuchsian then we recover Basmajian's original identity.
\end{proof}

\textit{\textbf{Remarks.}}   
(1)
In Theorem \ref{thm:identity}, $G_\rho$ is defined using the Frenet curve associated to the doubled representation $\hat{\rho}: \pi_1(\hat{\Sigma}) \to \PSL(n,\R)$. However, if we are given $\Sigma$ as a subsurface of a closed surface $S$ and a Hitchin representation $\rho': \pi_1(S) \to \PSL(n,\br)$,  then we may use the cross ratio associated to $\rho'$ restricted to $\pi_1(\Sigma)^{4*}$, which agrees with that of $\hat{\rho}$.

(2) Unlike the case for closed surfaces, $n$-Hitchin representations for $\Sigma$ do not fill a component of the $\PSL(n,\br)$-character variety.
We note that this identity can still hold for limits of Hitchin representations, similar to the hyperbolic setting.
For example, if $\rho_i$ is a sequence of Hitchin representations converging to a representation $\rho$ such that $\rho(\al_1)$  has all eigenvalues equal to 1 and $\rho(\gamma)$ is purely loxodromic for all nontrivial elements not conjugate to a power of $\al_1$, then by \eqref{eq:onecpt} we see that $G_{\rho_i}(x) \to G_\rho(x) = 0$ for all $x \in \cO_{1,j}(\Sigma, \cA)$. Thus, the identity extends to $\rho$.


\section{Relations to the McShane-Mirzakhani Identity}\label{sec:connections}

In this section, we discuss the relation between our identity and Labourie-Mcshane's generalization of the McShane-Mirzakhani identity.  We will first consider the hyperbolic surface case and then generalize to Hitchin representations.  

There are three spectral identities on hyperbolic surfaces with nonempty totally geodesic boundary (the McShane-Mirzkani \cite{mirzakhani}, Basmajian \cite{Basmajian:1993eu}, and Bridgeman \cite{Bridgeman:2011ffa} identities)  that originally appeared to be using completely different ideas, but were put into a unified framework by S.P. Tan by viewing them as different decompositions of the geodesic flow.  This viewpoint is outlined in the survey \cite{Bridgeman:2016tw}. These ideas led to the Luo-Tan identity  for closed surfaces \cite{Luo:2014wi}. This is the viewpoint we take in this section.

We note that finding relationships between the identities listed has been of recent interest.  Connections between Basmajian and Bridgeman's identities were explored in \cite{Bridgeman:2014bz} and \cite{Vlamis:2015is}.  Also, in a sense, the identity of Luo-Tan for closed surfaces gives connections between Bridgeman's identity and that of McShane-Mirzakhani.

The McShane-Mirzakhani identity gives the length of a boundary component as sum over a collection of  pairs of pants in the surface.  As the geometry of a pair of pants is dictated by the lengths of its boundary components, the summands depend on the lengths of simple closed geodesics in the surface.  In order to prove this identity, one has to give a decomposition of the boundary into intervals.  As this is the same idea for Basmajian's identity, the goal of this section is to relate the Basmajian decomposition of the boundary to that of the McShane-Mirzakhani decomposition.

\subsection{McShane-Mirzakhani Decomposition}
Let $F$ be a compact hyperbolic surface with nonempty totally geodesic boundary.  Fix $\al$ to be a component of $\partial F$.  For a point $x\in\al$, let $\be_x(t)$ be the geodesic obtained by flowing the unit vector $v_x$ normal to $\al$ at $x$ for time $t$.  Define $t_x \in \br_+$ to be either
\begin{itemize}
\item the first value of $t$ such that there exists $t_0\in [0, t)$ with $\be_x(t) = \be_x(t_0)$, i.e. $t_x$ is the first time the geodesic obtained by flowing $v_x$ hits itself, or
\item if the arc obtained from this flow is simple and returns to the boundary, then we let $t_x$ to be the time it takes to return to $\partial F$, i.e. $\be_x(t_x) \in \partial F$, or
\item if the arc is simple and infinite in length, let $t_x = \infty$.
\end{itemize}
Note that the set of boundary points with $t_x = \infty$ is measure zero as the limit set projects to a set of measure zero on $\al$ in the natural Lebesgue measure class.  For those $x \in  \al$ with $t_x < \infty$,
define the geodesic arc $\delta_x = \be_x([0,t_x])$. The arc $\delta_x$ defines a pair of pants $P_x$ as follows (there are two cases): 
\begin{itemize}
\item[(i)] If $\delta_x$ is simple and finite, let $\al'$ be the component of $\partial F$ containing $n_t(t_x)$  (possibly $\al'=\al$) and define $P_x$ to be the neighborhood of $\delta_x \cup\al \cup\al'$ with totally geodesic boundary.
\item[(ii)] If $\delta_x$ is not simple, then define $P_x$ to be the neighborhood of $\delta_x \cup \al$ with totally geodesic boundary.  This case is shown in Figure \ref{fig:pants}.
\end{itemize}

\begin{figure}[t]
\begin{center}\begin{overpic}[scale=0.5]{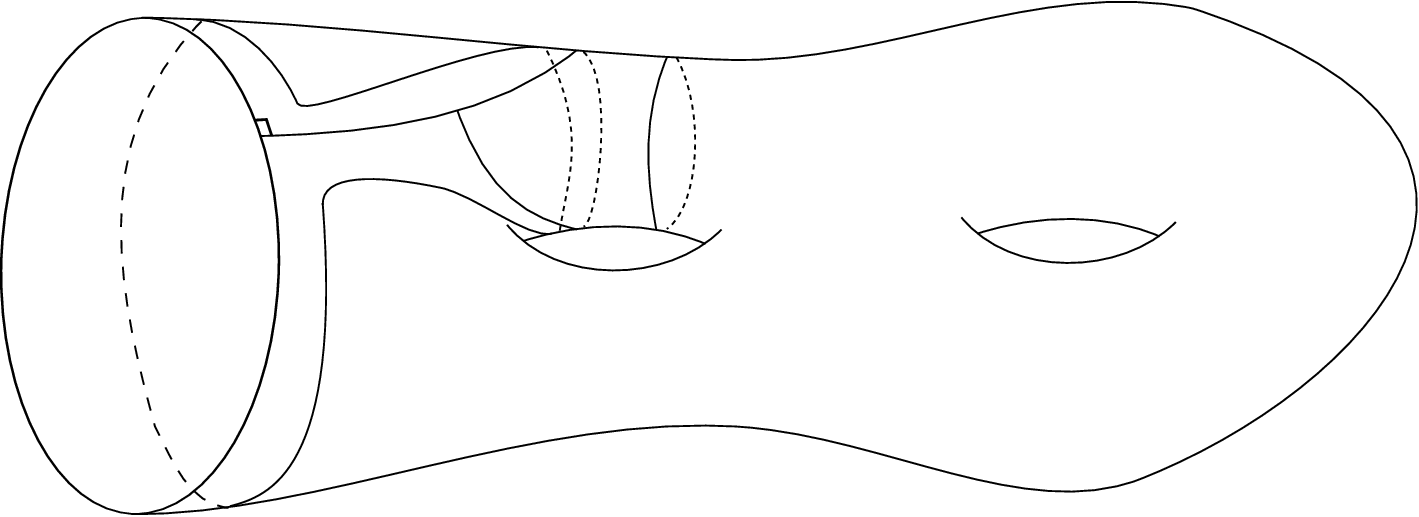}
\put(15,26){$x$}
\put(29,24.8){$\delta_x$}
\end{overpic}
\end{center}
\caption{An example of $P_x$ with $\delta_x$ non-simple.}
\label{fig:pants}
\end{figure}

The McShane-Mirzakhani decomposition of the boundary is as follows.  Let $\mathcal{P}_\al(F)$ be the set of embedded pairs of pants $P \subset F$ with geodesic boundary and with $\al$ as a boundary component.  For $P\in \mathcal{P}_\al(F)$ set
$$V_P = \{x \in \al \colon P_x = P\},$$
We then have that $V_P$ is a disjoint union of two intervals unless $P$ contains two components of $\partial F$, in which case $V_P$ is a single interval.
Further, $V_P\cap V_{P'} = \emptyset$ for $P\neq P'$ yielding
$$\ell(\al) = \ell\left(\bigcup_{P\in\mathcal{P_\al}(F)} V_P\right) = \sum_{P\in \mathcal{P}_\al(F)} \ell(V_P) \,,$$
see \cite{mirzakhani} or \cite{Bridgeman:2016tw} for details.
The McShane-Mirzakhani identity is derived from computing $\ell(V_P)$ for $P\in \mathcal{P}_\al(F)$.  

In the case that $F$ has a single boundary component, this identity becomes
$$\ell(\al) = \sum_{P\in \mathcal{P}_\al(F)} \log\left( \frac{e^{\frac{\ell(\partial P)}2}+e^{\ell(\partial F)}}{e^{\frac{\ell(\partial P)}2} +1}\right)\,.$$

\subsection{Comparing Decompositions}

In \S\ref{sec:projective}, we saw how to decompose the boundary for Basmajian's identity using orthogonal projection in the universal cover; let us give the same decomposition from a slightly different perspective that better matches the discussion on the McShane-Mirzakhani decomposition.  

For $x \in \partial F$, let $\be_x$ be the oriented geodesic obtained by flowing the vector normal to $\partial F$ based at $x$ as before.  $\be_x$ will have finite length and terminate in $\partial F$ for almost every $x\in \partial F$ as the limit set projects to a set of measure zero on $\partial F$.  For every orthogeodesic $\be \in \cO(F)$ we define
$$U_\be = \{x \in \partial F \colon \be_x \text{ is properly isotopic to } \be\}.$$
As no two orthogeodesics are properly isotopic, we see that $U_\be \cap U_{\be'} = \emptyset$ and as almost every $\be_x$ is properly isotopic to some orthogeodesic we again arrive at Basmajian's identity
$$\ell(\partial F) = \ell\left(\bigcup_{\be \in \cO(F)} U_\be \right) = \sum_{\be\in \cO(F)} \ell(U_\be) = \sum_{\be\in \cO(F)} 2 \log\coth \frac{\ell(\be)}2 \,.$$

As the McShane-Mirzakhani identity calculates the length of a particular boundary component, for $\al$ a component of $\partial F$, let $\cO_\al(F)$ be the collection of orthogeodesics emanating from $\al$. 

\begin{Prop}\label{prop:decomp}
Let $F$ be a compact hyperbolic surface with nonempty totally geodesic boundary. For each $\be \in \cO_\al(F)$, there exists $P\in \mathcal{P}_\al(F)$ such that $U_\be \subset V_P$.
\end{Prop}

\begin{proof}
There exists $x$ such that $\be = \be_x$, so we set $P = P_x$. Given $y\in U_\be$, we know that there is a proper isotopy taking $\be_y$ to $\be$, which must also take $\delta_y$ to $\delta_x$.  Given the definition of $P_y$, we have that $P_y = P$.  
\end{proof}

For $P \in \mathcal{P}_\al(F)$, let 
$$\mathcal{O}_P = \{ \be \in \cO_\al(F) \colon U_\be \subset V_P\} \,.$$
We then immediately have:
\begin{Cor}\label{cor:relation}
Let $F$ be a compact hyperbolic surface with nonempty totally geodesic boundary. For $P\in \mathcal{P}(F)$ 
$$\ell(V_P) = \sum_{\be \in \cO_P} 2 \log\coth \frac{\ell(\be)}2 \,.$$
\end{Cor}

\subsection{Decompositions in the Hitchin Setting}
In order to proceed, we need to translate the geometric language in the two decompositions to information about the fundamental groups of the surface.  We have already seen how to do this in the context of Basmajian's identity using the orthoset in \S\ref{sec:cosets}.  Now let us do the same for the McShane-Mirzakhani identity following \cite{Labourie:2009gb}. 

Let $\Sigma$ be a compact connected oriented surface with nonempty boundary whose double has genus at least two.  Fix a hyperbolic metric $\si$ on $\Sigma$ such that $\partial \Sigma$ is totally geodesic.  As we have done before, let us identify the universal cover of $\Sigma$ with a convex subset of $\bh^2$ cut out by geodesics. Fix a positive peripheral marking $\cA = \{\al_1, \ldots, \al_m\}$ for $\Sigma$ and let $\al = \al_1$ be a fixed peripheral element.  As in the previous section, we have the set $\mathcal{P}_\al(\Sigma, \si)$ consisting of embedded pairs of pants with totally geodesic boundary containing the component of $\partial \Sigma$ represented by $\al$.  We would like to replace these geometric objects with topological ones. In particular, we will translate $V_P$ into a subset of $S_\infty^1$ instead of a subset of $\al$ itself. In the geometric setting, this would be done via projection from $\al$ to $(\al^+, \al^-) \subset \Sp^1_\infty$.

Given $P \in \mathcal{P}_\al(\Sigma, \si)$ we can find a {\it good} pair $(\be,\gamma) \in \pi_1(P)^2$ such that $\al\gamma\be = e$ and $\be, \gamma$ oriented with $P$ on the left. Let $(\be',\gamma')$ be another good pair, then we will say that $(\be,\gamma) \sim (\be', \gamma')$ if for some $n$
\begin{align*}
\be' & = \al^n \be \al^{-n}\\
\gamma' & = \al^n \gamma \al^{-n} \,.
\end{align*}
Up to this equivalence there only exist two such pairs:  $(\be, \gamma)$ and $(\gamma, \gamma \be \gamma^{-1})$.  These pairs and equivalences depend only on the topology, so let us define $\mathcal{P}_\al(\Sigma)$ to be the set of isotopy classes of embedded pairs of pants in $\Sigma$ containing $\al$ as a boundary component.  Note that we have a natural bijection $\mathcal{P}_\al(\Sigma,\si) \to \mathcal{P}_\al(\Sigma)$ by sending $P$ to its isotopy class $[P]$.   

The pairs $(\be, \gamma)$ and $(\gamma, \gamma \be \gamma^{-1})$ correspond to the two isotopy classes of embeddings of a fixed pair of pants $P_0$ into $\Sigma$ with a choice of peripheral elements $\al_0, \beta_0, \gamma_0 \in \pi_1(P_0)$ with $P_0$ on the left, $\al_0 \gamma_0 \beta_0 = e$ and $\al_0 \mapsto \al$. This language is used in \cite{Labourie:2009gb}.

\begin{figure}[t]
\begin{center}\begin{overpic}[scale=0.3]{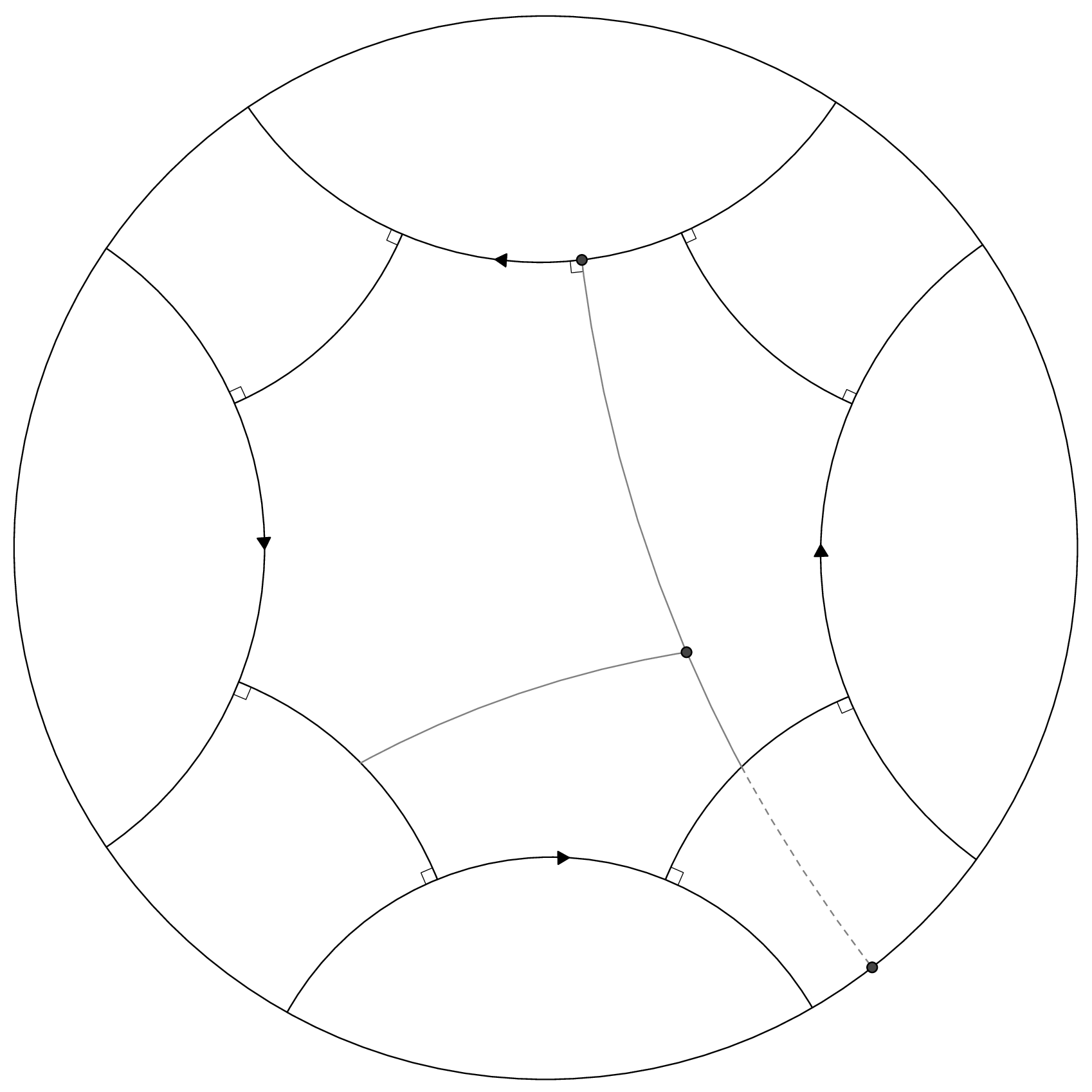}
\put(20,92){$\al^+$}
\put(77.5,91.5){$\al^-$}
\put(5,77){$\be^-$}
\put(3,20){$\be^+$}
\put(91,20){$\gamma \cdot \be^- = (\gamma \be \gamma^{-1})^-$}
\put(91,77){$\gamma \cdot \be^+ = (\gamma \be \gamma^{-1})^+$}
\put(71,4){ $\gamma^+$}
\put(23,3){$\gamma^-$}
\put(52,78){$\tilde x$}
\put(51,60){$\tilde\delta_x$}
\put(40,50){$D$}
\put(81,7){$\tilde\delta_x^+$}
\end{overpic}
\end{center}
\caption{A fundamental domain $D$ for $P$ and the lift of $\delta_x$. One can verify that $\al \gamma \beta = e$ and $\al^{-1} \cdot \beta^\pm = \gamma \cdot \beta^\pm$.}
\label{fig:hyppants}
\end{figure}
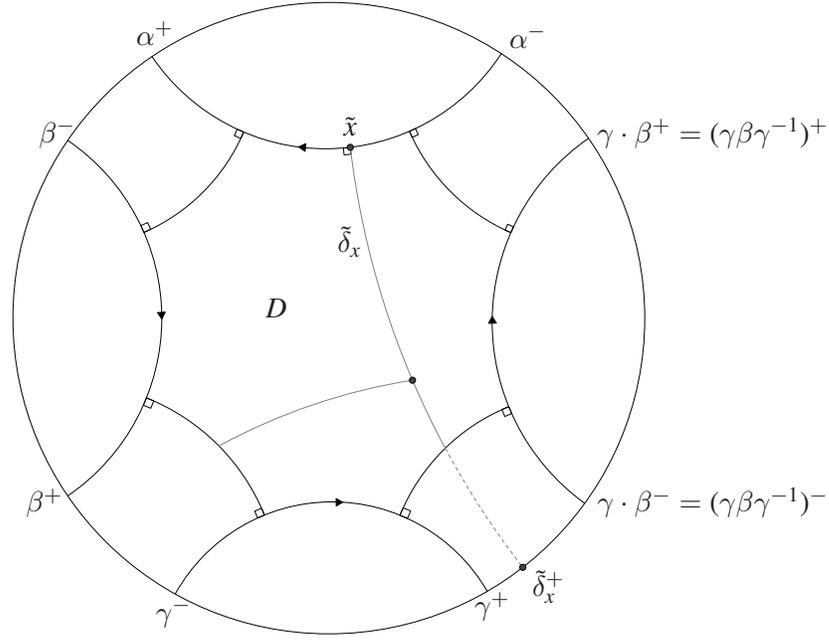

Let us fix $[P]\in \mathcal{P}_\al(\Sigma)$ with $P \in \mathcal{P}_\al(\Sigma,\si)$ and fix a good pair $(\be,\gamma)\in \pi_1(P)^2 \subset \pi_1(\Sigma)^2$.
We can draw a fundamental domain $D$ for $P$ as in Figure \ref{fig:hyppants}. 
Abusing notation and letting $\al$ also denote its geodesic representative in $\partial \Sigma$, let $x \in \al$ be such that $P_x = P$.
Lift $x$ to $\tilde x\in D$ on the geodesic $\mathfrak{G}(\al^-, \al^+) \subset \bh^2$ and let $\tilde\delta_x$ be the lift of $\delta_x$ (as defined in the previous subsection) living in this fundamental domain. 

Assuming $\be$ and $\gamma$ are not peripheral, observe that $\delta_x$ determines $P$ if and only if $\delta_x \subset P$ and has finite length, see Figure \ref{fig:pants} for an example. In particular, this means that $\delta_x$ stays inside $P$ and either self intersects or hits $\al$. This is equivalent to having $\tilde\delta_x^+$ in the set $$\tilde J_{P} =(\be^+, \gamma^-)  \cup (\gamma^+, \gamma \cdot \be^-) \subset \Sp^1_\infty$$ as shown in Figure \ref{fig:hyppants}. The orthogonal projection of $J_{P}$  to the geodesic $\mathfrak{G}(\al^-, \al^+) \subset \bh^2$ followed by the universal covering projection to $\partial \Sigma$ is injective and corresponds to  $V_{P}$. 

Now suppose only $\gamma$ is peripheral, then $\delta_x$ determines $P$ if and only if $\tilde \delta_x^+$ is in the interval $\tilde J_{P}= (\be^+, \gamma \cdot \be^-)$. We simply add in the interval $(\gamma^-, \gamma^+)$ to allow for simple arcs $\delta_x$ that hit the boundary component $\gamma$ for the scenario in the previous paragraph.

Similarly, if only $\be$ is peripheral, then $ \delta_x$ determines $P$ if and only if $\tilde\delta_x^+$ is some $\al$ translate of a point in the interval $$\tilde J_{P} = (\al \cdot \gamma^+, \gamma^-) = (\beta^-, \gamma^-) \cup \, \al \cdot (\gamma^+, \gamma \cdot \be^-) \, .$$ The technicality of translating by $\al$ arises because we chose our lift $\tilde x \in D$ and want to write $\tilde J_{P}$ as one interval.

If both $\be$ and $\gamma$ are peripheral, then $\Sigma = P$ and the interval is simply $J_{P} = (\be^-, \gamma \cdot \be^-)$.  The same sequence of projections also gives $V_{P}$ in these cases.

Let $\rho$ be a Hitchin representation of $\Sigma$ and let $B = B_{\rho}$ be the associated cross ratio.  We define the \textit{pants gap function} $H_\rho: \mathcal{P}_\al(\Sigma) \to \br$ as follows. Let $[P]\in \mathcal{P}_\al(\Sigma)$ and let $(\be, \gamma)\in \pi_1(P)^2\subset\pi_1(\Sigma)$ be a good pair. Define the auxiliary function $i_{\partial} : \pi_1(\Sigma) \to \{0,1\}$ by $i_\partial(\omega) = 1$ if $\omega$ is primitive peripheral and $i_\partial(\omega) = 0$ otherwise. Then
\begin{align*} H_\rho([P])&  = \log\left[B(\al^+, \gamma^-, \al^-, \be^+)\cdot B(\al^+, \gamma \cdot \beta^-, \al^-, \gamma^+)\right] + \\
& + i_{\partial}(\beta) \log B(\al^+, \beta^+, \al^-, \beta^-) + i_{\partial}(\gamma) \log B(\al^+, \gamma^+, \al^-, \gamma^-) \, .
\end{align*}
With this setup at hand, the McShane-Mirzakhani identity for Hitchin representations from \cite{Labourie:2009gb} states 
$$\ell_\rho(\al) = \sum_{[P]\in \mathcal{P}_\al(\Sigma)} H_\rho([P]) \,.$$
If we let $\mathbb{T} = \br / \ell_\rho(\al)\bz$ and let $J_P$ be the projection of $\tilde J_{P}$ under the composition of the projection $\pi : \br \to \mathbb{T}$ and the map $F_\rho$ defined in \eqref{eq:projection}, then the McShane-Mirzakhani identity is saying that the $J_P$ are all disjoint and give a full measure decomposition of $\mathbb{T}$. 

As in the proof of Theorem \ref{thm:identity}, for a primitive peripheral element $\be$ of $\pi_1(\Sigma)$, let $\tilde I_\be = (\be^-, \be^+)$ and $I_\be = \pi((F_\rho(\be^-), F_\rho(\be_+)))$. Using $\al = \al_1$ in some positive peripheral marking, let
$$\cO_\al(\Sigma, \cA) = \cO_{1,j}(\Sigma,\cA )\, .$$
We saw that $I_\be$ corresponds to an element $x \in \cO_\al(\Sigma, \cA)$, so let us rename this interval $I_x$.  As the sets $J_P$ are all disjoint, it follows that for $x\in \cO_\al(\Sigma, \cA)$, there is a unique $[P]\in \mathcal{P}_\al(\Sigma)$ such that $I_x \subset J_P$. This gives the analog of Proposition \ref{prop:decomp}:

\begin{Prop}
Let $\Sigma$ be a compact connected orientable surface with nonempty boundary whose double has genus at least 2.  For each $x\in \cO_\al(\Sigma, \cA)$ there exists a unique $[P]\in \mathcal{P}_\al(\Sigma)$ such that $I_x \subset J_P$.
\end{Prop}

For $[P]\in \mathcal{P}_\al(\Sigma)$, let 
$$\cO_P(\Sigma, \cA) = \{x \in \cO_\al(\Sigma, \cA) \colon I_x \subset J_P\} \,.$$
We then immediately have the analog of Corollary \ref{cor:relation}:

\begin{Cor}
Let $\Sigma$ be a compact connected orientable surface with nonempty boundary whose double has genus at least 2 and let $[\rho]$ a Hitchin representation of $\pi_1(\Sigma)$.  For $[P]\in \mathcal{P}_\al(\Sigma)$
$$H_\rho([P]) = \sum_{x\in \cO_P(\Sigma, \cA)} G_\rho(x) \,.$$
\end{Cor}

\footnotesize 
\bibliographystyle{amsalpha}	
\bibliography{references}

\end{document}